\documentclass[reqno]{amsproc}
\usepackage[utf8]{inputenc}
\usepackage{orcidlink}
\usepackage{amsfonts}
\usepackage{graphicx}
\usepackage{verbatim}
\usepackage{amscd}
\usepackage{amsmath, physics}
\usepackage{enumitem}
\usepackage{amssymb}
\usepackage{latexsym}
\usepackage{hyperref}
\usepackage[all]{xy}
\usepackage{color}
\usepackage{mathrsfs}
\newtheorem{theorem}{\bf Theorem}

\newtheorem{lemma}{\bf Lemma}
\theoremstyle{definition}

\usepackage{empheq,etoolbox}
\usepackage{stmaryrd}
\numberwithin{equation}{section}
\patchcmd{\subequations}
  {\theparentequation\alph{equation}}
  {\theparentequation.\alph{equation}}
  {}{}

\begin{document}

\title[Rigidity of gradient K\"ahler--Ricci solitons]{Rigidity of nonsteady gradient
K\"ahler--Ricci solitons with constant scalar curvature}

\author[M. A. R. M. Hor\'acio]{Matheus Andrade Ribeiro de Moura Hor\'acio}
\address{Departamento de Matem\'atica, Universidade de Bras\'ilia,
70910-900 Bras\'ilia, DF, Brazil}
\email{matheus.andrade5488@gmail.com}

\subjclass[2020]{Primary 53E20; Secondary 53C55, 53C25}
\keywords{Gradient Ricci soliton, K\"ahler geometry, constant scalar
curvature, rigidity}
\date{}

\hypersetup{
  pdftitle={Rigidity of nonsteady gradient K\"ahler--Ricci solitons with constant scalar curvature},
  pdfauthor={Matheus Andrade Ribeiro de Moura Hor\'acio}
}

\maketitle{}

\begin{abstract}
We prove that every complete nonsteady gradient K\"ahler--Ricci soliton
with constant scalar curvature is rigid. This establishes Cao's rigidity
conjecture for K\"ahler--Ricci solitons in arbitrary dimension and gives
the corresponding result for expanding solitons, without additional
curvature assumptions. The proof uses a rigidity criterion for gradient
Ricci solitons with constant scalar curvature, expressed by the vanishing
of the Lie derivative of the Ricci tensor along the soliton vector field.
In the K\"ahler case, this vanishing follows from constant scalar curvature,
the closedness of the Ricci form, and the classical real holomorphicity
of the soliton vector field.
\end{abstract}
\maketitle

\section{Introduction and main results}

A gradient Ricci soliton $(M^n,g,f,\lambda)$ is a Riemannian
manifold $(M^n,g)$ with a smooth function $f$ and a constant $\lambda$
satisfying
\begin{equation}\label{soliton}
\mathrm{Ric}+\nabla^2f=\lambda g,
\end{equation}
where $\mathrm{Ric}$ is the Ricci tensor and $\nabla^2f$ is the Hessian of
$f$. The soliton is shrinking, steady, or expanding according as
$\lambda>0$, $\lambda=0$, or $\lambda<0$, respectively. Throughout this
article, nonsteady means that $\lambda\ne0$.

Following Petersen and Wylie \cite{PW}, a gradient Ricci soliton is
called rigid if it is a flat vector bundle
$N\times_\Gamma\mathbb{R}^k$, where $N$ is Einstein with Einstein
constant $\lambda$ and the discrete group $\Gamma$ acts freely and
properly discontinuously by isometries on $N$ and orthogonally on
$\mathbb{R}^k$. The potential function is, up to an additive constant,
$f=\frac{\lambda}{2}d^2$, where $d$ is the distance in each flat fiber
to the zero section.
In particular, a rigid soliton has constant scalar curvature.
On a compact soliton, the converse follows by tracing and integrating
\eqref{soliton}: constant scalar curvature forces $f$ to be constant,
so the metric is Einstein. The question is therefore whether constant
scalar curvature forces rigidity in the noncompact case.

Cao conjectured that every complete shrinking gradient Ricci soliton
with constant scalar curvature is rigid; see Cheng and Zhou \cite{ChZ}.
Cheng and Zhou proved the four-dimensional case without a K\"ahler
assumption.
More recently, Ou, Qu and Wu \cite{OQW} and Wang and Wu \cite{WW}
obtained further rigidity results for shrinking solitons with constant
scalar curvature, including new proofs of the four-dimensional case and
higher-dimensional results under additional assumptions. In the K\"ahler setting, Fern\'andez-L\'opez and
Garc\'ia-R\'io \cite[Theorem~4]{FG} had established rigidity in real dimensions
$4$ and $6$. Their proof uses the even multiplicities of the Ricci
eigenvalues and a rigidity theorem for solitons with at most three
distinct Ricci eigenvalues.

Rigidity results in higher dimensions have also been obtained under
additional geometric assumptions. Chen and Zhu \cite{CZ} proved that
complete shrinking or expanding gradient K\"ahler--Ricci solitons with
harmonic Bochner tensor are rigid. Zhang \cite[Theorem~1.2]{Zhang}
proved rigidity for complete noncompact gradient K\"ahler--Ricci
shrinkers with constant scalar curvature in complex dimension $m=5$
or $m\ge7$, assuming that the fourth-order divergence of the curvature
tensor or of the Bochner tensor is nonnegative.

We prove that constant scalar curvature suffices in the K\"ahler setting
in every dimension, for both shrinking and expanding solitons.

\begin{theorem}\label{kahler}
Every complete nonsteady gradient K\"ahler--Ricci soliton with constant
scalar curvature is rigid.
\end{theorem}

More explicitly, if $M$ has complex dimension $m$, its universal cover
is holomorphically isometric to
\[
N^k\times\mathbb{C}^{m-k},\qquad 0\le k\le m,
\]
where $N^k$ is a complete simply connected K\"ahler--Einstein manifold
of complex dimension $k$ with Einstein constant $\lambda$; $N^0$ denotes
a point. Thus $M$ is a quotient of this product by a discrete group of
holomorphic isometries acting freely and properly discontinuously.

The shrinking case proves Cao's conjecture for K\"ahler--Ricci solitons.
The theorem removes the dimension restriction from the nonsteady case
of Fern\'andez-L\'opez and Garc\'ia-R\'io's result without additional
curvature assumptions.

The proof rests on a criterion that applies without a K\"ahler assumption.
Petersen and Wylie \cite[Theorem~1.2]{PW} characterized rigidity of
complete nonsteady gradient Ricci solitons by constant scalar curvature
and radial flatness, namely the vanishing of sectional curvature on
planes containing $\nabla f$ at regular points of $f$. Under constant
scalar curvature, they also obtained radial flatness from
$0\le\mathrm{Ric}\le\lambda g$ for
shrinkers and $\lambda g\le\mathrm{Ric}\le0$ for expanders
\cite[Proposition~1.3(2)]{PW}. We obtain these inequalities from a
condition on the Lie derivative of the Ricci tensor.

\begin{theorem}\label{crit}
Let $(M^n,g,f,\lambda)$ be a complete nonsteady gradient Ricci soliton
with constant scalar curvature. Then $M$ is rigid if and only if
\begin{equation}\label{stationary}
\mathcal L_{\nabla f}\mathrm{Ric}=0,
\end{equation}
where $\mathrm{Ric}$ denotes the covariant $(0,2)$-tensor field.
\end{theorem}

Fern\'andez-L\'opez and Garc\'ia-R\'io \cite[Remark~17]{FG} showed
that rigidity is equivalent to constancy of the traces of the first four
powers of the Ricci operator. They also observed that constant scalar
curvature $R$ makes the second trace constant, whereas the third and fourth were
``not yet completely understood.'' In the proof of Theorem~\ref{crit},
equation~\eqref{stationary} supplies the relations between these traces
that give $\mathrm{Ric}^2=\lambda\mathrm{Ric}$, and the results of
Petersen and Wylie complete the argument.

For Theorem~\ref{kahler}, constant scalar curvature gives
$\mathrm{Ric}(\nabla f)=0$. The Ricci form $\rho$ is closed and its contraction
with $\nabla f$ vanishes, so Cartan's formula gives
$\mathcal L_{\nabla f}\rho=0$. The classical real holomorphicity of
$\nabla f$ then gives \eqref{stationary}. Equations~\eqref{transport}
and~\eqref{pj} therefore yield
$\mathrm{tr}\,\mathrm{Ric}^3=\lambda^2R$ and
$\mathrm{tr}\,\mathrm{Ric}^4=\lambda^3R$ in the K\"ahler case.

\section{A rigidity criterion}

We denote the scalar curvature by $R$ and use $\mathrm{Ric}$ also for
the self-adjoint endomorphism defined by the Ricci tensor and the metric.
Powers and traces of $\mathrm{Ric}$ refer to this endomorphism, whereas
Lie derivatives of $\mathrm{Ric}$ refer to the covariant tensor.
The soliton equation can then be written as
\begin{equation}\label{operatorsoliton}
\mathrm{Ric}+\nabla(\nabla f)=\lambda I,
\end{equation}
where $I$ is the identity endomorphism. We first recall the identities
that will be used in the proof.

\begin{lemma}\label{identities}
Let $(M^n,g,f,\lambda)$ be a gradient Ricci soliton with constant scalar
curvature $R$. Then
\begin{equation}\label{trRic}
\mathrm{Ric}(\nabla f)=0,\qquad \mathrm{tr}\,\mathrm{Ric}=R,\qquad
\mathrm{tr}\,\mathrm{Ric}^2=\lambda R.
\end{equation}
\end{lemma}

\begin{proof}
The standard soliton identities \cite[Lemma~2.5]{PW} give
\begin{equation}\label{scalaridentities}
\nabla R=2\mathrm{Ric}(\nabla f),\qquad
\frac12\Delta_fR=\lambda R-\mathrm{tr}\,\mathrm{Ric}^2,
\end{equation}
where $\Delta_f=\Delta-\langle\nabla f,\nabla\cdot\rangle$.
Since $R$ is constant, the first and third assertions follow. The second
is the definition of scalar curvature.
\end{proof}

\begin{proof}[Proof of Theorem~\ref{crit}]
Assume first that \eqref{stationary} holds. For any vector fields $X$
and $Y$, the Lie derivative of the Ricci tensor satisfies
\[
\begin{aligned}
(\mathcal L_{\nabla f}\mathrm{Ric})(X,Y)
&=(\nabla_{\nabla f}\mathrm{Ric})(X,Y)\\
&\quad+\mathrm{Ric}(\nabla_X\nabla f,Y)
  +\mathrm{Ric}(X,\nabla_Y\nabla f).
\end{aligned}
\]
By \eqref{operatorsoliton}, we have
$\nabla_X\nabla f=(\lambda I-\mathrm{Ric})X$. Self-adjointness of
$\mathrm{Ric}$ therefore gives
\[
\begin{aligned}
\mathrm{Ric}(\nabla_X\nabla f,Y)
&=\mathrm{Ric}(X,\nabla_Y\nabla f)\\
&=\langle(\lambda\mathrm{Ric}-\mathrm{Ric}^2)X,Y\rangle.
\end{aligned}
\]
Since the connection is metric compatible, condition~\eqref{stationary}
is equivalent to the following identity for the Ricci endomorphism:
\begin{equation}\label{transport}
\nabla_{\nabla f}\mathrm{Ric}=2(\mathrm{Ric}^2-\lambda\mathrm{Ric}).
\end{equation}

To determine the Ricci eigenvalues, set
\[
p_j=\mathrm{tr}\,\mathrm{Ric}^j=\sum_{i=1}^n R_i^j,
\qquad j\ge1,
\]
where $R_1,\ldots,R_n$ are the eigenvalues at a point, counted with
multiplicity. Differentiating the trace and using its invariance under
cyclic permutations, we obtain
\[
\begin{aligned}
\nabla_{\nabla f}p_j
&=\sum_{a=0}^{j-1}\mathrm{tr}\bigl(
\mathrm{Ric}^a(\nabla_{\nabla f}\mathrm{Ric})
\mathrm{Ric}^{j-1-a}\bigr)\\
&=j\,\mathrm{tr}\bigl(
\mathrm{Ric}^{j-1}\nabla_{\nabla f}\mathrm{Ric}\bigr).
\end{aligned}
\]
Substituting \eqref{transport} yields
\begin{equation}\label{pj}
\nabla_{\nabla f}p_j=2j(p_{j+1}-\lambda p_j).
\end{equation}
By Lemma~\ref{identities}, $p_2=\lambda R$ is constant. Taking $j=2$
in \eqref{pj}, we obtain $p_3=\lambda p_2$. Hence $p_3$ is also
constant, and the same equation with $j=3$ gives $p_4=\lambda p_3$.
Thus $p_3=\lambda^2R$ and $p_4=\lambda^3R$, determining the higher
traces discussed in \cite[Remark~17]{FG}.

Fern\'andez-L\'opez and Garc\'ia-R\'io considered the expression
$\displaystyle\sum_i R_i^2(\lambda-R_i)^2$ in their characterization of rigidity
by constant traces of the first four powers of the Ricci operator.
Here the relations above give
\[
\begin{aligned}
|\mathrm{Ric}^2-\lambda\mathrm{Ric}|^2
&=\sum_{i=1}^n R_i^2(\lambda-R_i)^2\\
&=p_4-2\lambda p_3+\lambda^2p_2=0.
\end{aligned}
\]
It follows that $\mathrm{Ric}^2=\lambda\mathrm{Ric}$, and thus every
Ricci eigenvalue belongs to $\{0,\lambda\}$. Consequently,
\[
\begin{cases}
0\le\mathrm{Ric}\le\lambda g,&\lambda>0,\\
\lambda g\le\mathrm{Ric}\le0,&\lambda<0.
\end{cases}
\]
Since $M$ is complete and $R$ is constant, Theorem~1.2 and
Proposition~1.3(2) of Petersen and Wylie \cite{PW} imply that $M$ is
rigid.

Conversely, suppose that $M$ is rigid. On the product cover
$N\times\mathbb{R}^k$, the Ricci tensor is parallel and its
endomorphism is $\lambda I$ on $TN$ and zero on the Euclidean factor.
Hence $\nabla_{\nabla f}\mathrm{Ric}=0$ and
$\mathrm{Ric}^2=\lambda\mathrm{Ric}$. The Lie derivative formula
above gives \eqref{stationary} on the cover and therefore on $M$.
\end{proof}

\section{The K\"ahler case}

We now verify condition~\eqref{stationary} for a gradient
K\"ahler--Ricci soliton with constant scalar curvature.

\begin{proof}[Proof of Theorem~\ref{kahler}]
Let $J$ be the complex structure. The Ricci endomorphism commutes with
$J$. Since $\nabla J=0$, equation~\eqref{operatorsoliton} gives, for every
vector field $Y$,
\[
\begin{aligned}
(\mathcal L_{\nabla f}J)Y
&=J(\nabla_Y\nabla f)-\nabla_{JY}\nabla f\\
&=J(\lambda I-\mathrm{Ric})Y-(\lambda I-\mathrm{Ric})JY=0.
\end{aligned}
\]
Thus $\nabla f$ is real holomorphic.

Consider the Ricci form $\rho$, with the convention
\[
\rho(X,Y)=\mathrm{Ric}(JX,Y).
\]
Lemma~\ref{identities} gives $\mathrm{Ric}(\nabla f)=0$. Hence
\[
(\iota_{\nabla f}\rho)(Y)
=\mathrm{Ric}(J\nabla f,Y)
=\langle J\mathrm{Ric}(\nabla f),Y\rangle=0.
\]
The Ricci form is closed. Cartan's formula then gives
\[
\mathcal L_{\nabla f}\rho
=\mathrm{d}(\iota_{\nabla f}\rho)
 +\iota_{\nabla f}(\mathrm{d}\rho)=0.
\]
Taking the Lie derivative of the identity defining $\rho$, we obtain
\[
\begin{aligned}
(\mathcal L_{\nabla f}\rho)(X,Y)
&=(\mathcal L_{\nabla f}\mathrm{Ric})(JX,Y)
 +\mathrm{Ric}((\mathcal L_{\nabla f}J)X,Y)\\
&=(\mathcal L_{\nabla f}\mathrm{Ric})(JX,Y).
\end{aligned}
\]
The left-hand side vanishes and $J$ is invertible. Therefore
$\mathcal L_{\nabla f}\mathrm{Ric}=0$, and Theorem~\ref{crit}
implies that $M$ is rigid.
On the universal cover, the Einstein and Euclidean factors are tangent
to the $\lambda$- and $0$-eigenspaces of the Ricci endomorphism.
These are $J$-invariant, so the parallel complex structure splits with
the metric, giving the stated holomorphic product.
\end{proof}

\section*{Acknowledgments}

This study was financed by the Coordena\c{c}\~ao de Aperfei\c{c}oamento
de Pessoal de N\'ivel Superior -- Brasil (CAPES) -- Finance Code 001.

The author thanks the Departamento de Matem\'atica of the Universidade de
Bras\'ilia, where this work was carried out. He is particularly grateful
to his doctoral advisor, Jo\~ao Paulo dos Santos, for his valuable guidance,
patience, and support throughout his doctoral work. He also thanks Valter
Borges Sampaio Junior for his generosity with his time and for the many
mathematical discussions that have contributed to the author's mathematical
development and led to fruitful collaborations on related projects.
Their discussions of K\"ahler geometry have also broadened his understanding
of the subject.


\begin{thebibliography}{9}

\bibitem{CZ}
Q. Chen and M. Zhu,
\textit{On rigidity of gradient K\"ahler-Ricci solitons with harmonic
Bochner tensor},
Proc. Amer. Math. Soc. \textbf{140} (2012), no.~11, 4017--4025.

\bibitem{ChZ}
X. Cheng and D. Zhou,
\textit{Rigidity of four-dimensional gradient shrinking Ricci solitons},
J. reine angew. Math. \textbf{802} (2023), 255--274.

\bibitem{FG}
M. Fern\'andez-L\'opez and E. Garc\'ia-R\'io,
\textit{On gradient Ricci solitons with constant scalar curvature},
Proc. Amer. Math. Soc. \textbf{144} (2016), no.~1, 369--378.

\bibitem{OQW}
J. Ou, Y. Qu and G. Wu,
\textit{Some rigidity results on shrinking gradient Ricci soliton},
J. Geom. Anal. \textbf{35} (2025), no.~10, Paper No.~311, 14~pp.

\bibitem{PW}
P. Petersen and W. Wylie,
\textit{Rigidity of gradient Ricci solitons},
Pacific J. Math. \textbf{241} (2009), no.~2, 329--345.

\bibitem{WW}
C. Wang and G. Wu,
\textit{A note on rigidity of shrinking gradient Ricci solitons with
constant scalar curvature},
Results Math. \textbf{81} (2026), Paper No.~151.

\bibitem{Zhang}
L. Zhang,
\textit{On gradient shrinking and expanding K\"ahler--Ricci solitons},
Mediterr. J. Math. \textbf{19} (2022), Paper No.~15, 12~pp.

\end{thebibliography}
\end{document}